\documentclass[a4paper,11pt]{article}      

\pdfoutput=1

\usepackage{amsmath,amsthm}
\usepackage{amssymb}
\usepackage{amsbsy}
\usepackage{graphics}
\usepackage{pifont}
\usepackage{epsfig}
\usepackage{setspace}
\usepackage{graphicx}
\usepackage{enumerate}

\usepackage{geometry}
\usepackage[small]{titlesec}	

\usepackage{parskip}
\makeatletter
\def\thm@space@setup{%
  \thm@preskip=\parskip \thm@postskip=0pt
}\makeatother

\newcommand{\innerprod}[2]{\left\langle{#1},{#2}\right\rangle}
\newcommand{\set}[2]{\left\{ #1\ \left| \ #2 \right. \right\}}
\newcommand{\norm}[1]{\lVert#1\rVert}
\newcommand{\tpose}{^{\text{\scalebox{0.75}{$\top$}}\!}}

\usepackage{xcolor,calc}

\newcommand{\bs}[1]{\boldsymbol{#1}}

\title{\Large
Observer design for systems with an energy-preserving non-linearity\thanks{Preliminary versions of some results in this paper have been accepted to appear in the 50th IEEE Conference on Decision and Control, 12th-15th December 2011, Orlando, Florida, USA.  }}

\author{\large Andrew Wynn and Paul Goulart}

\newtheorem{thm}{Theorem}[section]
\newtheorem{lem}[thm]{Lemma}
\newtheorem{prop}[thm]{Proposition}

\newtheorem{alg}[thm]{Algorithm}
\newtheorem{rem}[thm]{Remark}

\begin{document}
\maketitle
\thispagestyle{empty}
\pagestyle{empty}

\begin{abstract}
Observer design is considered for a class of non-linear systems whose non-linear part is energy preserving.  A strategy to construct convergent observers for this class of non-linear system is presented. The approach has the advantage that it is possible, via convex programming, to prove whether the constructed observer converges, in contrast to several existing approaches to observer design for non-linear systems. Finally, the developed methods are applied to the Lorenz attractor and to a low order model for shear fluid flow. 
\end{abstract}

\section{Introduction}

Observer design for non-linear systems is an important and difficult problem. In this paper, observer design is considered for systems whose non-linear part has an energy preserving structure. In particular, 
\begin{equation} \label{nonlinsys}
\begin{array}{cclc} \dot x(t) &= &	Ax(t) + N(x(t))x(t), 	& t \geq 0, \\ 
												y(t) &= &	Cx(t), 							 	& t \geq 0, \\
												x(0) &= &	x_0 \in \mathbb{R}^n. &
\end{array}												
\end{equation}
where $A \in \mathbb{R}^{n\times n}$, $C \in \mathbb{R}^{p \times n}$ and $N:\mathbb{R}^n \rightarrow \mathbb{R}^{n \times n}$ is a linear operator. Furthermore, it is assumed that the non-linearity $N(x)x$ has the {\em energy preserving} property
\begin{equation} \label{energy}
x\tpose N(x)x = 0, \qquad x \in \mathbb{R}^n. 
\end{equation}

Observer design for non-linear systems has received much attention, with approaches falling into two main categories.  One approach, first considered in \cite{Krener} and generalized in \cite{Astolfi,Kazantzis,KrenerXiao,KrenerXiao2}, is to apply a change of co-ordinates to linearize the system, up to an additional term involving the output $y(t)$. Subsequently, linear design methods can be applied to create an observer for the transformed system, then the co-ordinate transformation is inverted to form an observer for the original, non-linear, system. The main drawback of this approach is that it is usually impossible to prove that the chosen co-ordinate transformation is invertible. Hence, while this is a powerful technique for observer design, it is difficult to prove in practice that the constructed observer will actually converge. 

The second approach is to assume a Lipschitz-type bound on the non-linear part of the system. For example, in addition to the standard Lipschitz assumption \cite{Aboky,Rajamani}, one-sided Lipschitz conditions \cite{Hu,XuHu} and a `less conservative' Lipschitz condition \cite{Rajamani2} have been studied. These techniques apply a Luenberger-type observer and require that the non-linearity is `small enough' with respect to the linear part of the dynamics. A major drawback of this approach is that systems with a dominant non-linear term often have a large Lipschitz bound and, if this is the case, it is unlikely to be possible to prove that a given observer converges. 

The difficulties of the above techniques arise either from excessive generality or overly restrictive assumptions. The co-ordinate transformation technique may theoretically be applied to {\em any} non-linear system, and is therefore unlikely to succeed in every case. For the Lipschitz approaches, a small global Lipschitz bound restricts the class of systems to which the results may be applied. For this reason, we aim for an approach to observer design that sits between these two extremes by only considering the particular class of non-linear system \eqref{nonlinsys} whose non-linear part satisfies \eqref{energy}.

The importance of dynamical systems of the form \eqref{nonlinsys}--\eqref{energy} is that they often arise in finite-dimensional approximations of non-linear physical systems, for example, the Navier-Stokes equations in fluid flows \cite{Berkooz} and the non-linear oscillation of beams in structural dynamics \cite{Hodges}. In experimental practice, such approximations are referred to as `low order models' and can be created directly from experimental data by using, for example, the method of Proper Orthogonal Decomposition \cite{Ravindran2,Ravindran}. From a theoretical perspective, it is therefore of great interest to study the control theoretic properties of such systems, with a view to guiding experimental implementation. The link between the Navier-Stokes equations and \eqref{nonlinsys} is presented in Section \ref{examplesSec}.

\textbf{Notation}: The $n$-sphere in $\mathbb{R}^{n+1}$ is defined as $\mathbb{S}^n := \set{x\in \mathbb{R}^n}{\norm{x}_2 = 1}$.   A matrix $P\in \mathbb{R}^{n \times n}$ is said to be {\em positive definite} (written $P \succ 0$) if its symmetric part satisfies $x\tpose(P+P\tpose)x > 0$, for any $x \in \mathbb{R}^n$, and {\em negative definite} if $-P$ is positive definite (written $P \prec 0$). The set of symmetric matrices of dimension $n$ is denoted $\mathcal{S}^n$. For matrices $A, B$ and $C$ of appropriate sizes, the shorthand 
\begin{equation*}
\left[ \begin{array}{cc} A & B \\ (\ast) & C \end{array} \right] := \left[ \begin{array}{cc} A & B \\ B\tpose & C \end{array} \right] 
\end{equation*}
is used to simplify the block matrix. For $r>0$ and $d \in \mathbb{R}^n$, the closed $\|\cdot\|_2$-norm ball centered at $d$ of radius $r$ is denoted 
\[
B_r(d):= \set{ x \in \mathbb{R}^n}{ \|x-d\|_2 \leq r}. 
\] 
For $i=1,\ldots,n$,
\begin{equation*}
\bs{e}_i = \underbrace{(0,\ldots0,1,0,\ldots,0)\tpose}_{1 \; \text{in} \; i^{\text{th}}\; \text{entry}}
\end{equation*}
denotes the $i^{\text{th}}$ element of the standard basis of $\mathbb{R}^n$. For sets $S,T \subset \mathbb{R}^n$ and $\alpha \in \mathbb{R}$,
\begin{equation*}
S\oplus T:=\{ s+ t: s \in S, t \in T\}, \quad \alpha S := \{ \alpha s : s \in S \}. 
\end{equation*}




\section{Observer Design} \label{obsDesSec}

The approach taken in this paper is to exploit and energy preserving properties of the non-linearity in \eqref{nonlinsys} to obtain a method for constructing a convergent observer. In particular, for a given gain matrix $L \in \mathbb{R}^{n \times p}$, the observer $(\hat x(t))_{t \geq 0}$ is assumed to have dynamics
\begin{equation} \label{LuOb}
\dot{\hat x} = A \hat{x} + N(\hat x) \hat x - L(y - C\hat x), \qquad \hat x (0) = \hat x_0 \in \mathbb{R}^n.
\end{equation}
Therefore, the observer error $e:=x-\hat x$ satisfies
\begin{equation} \label{Err}
\dot e = (A+LC)e + N(x)x - N(\hat x)\hat x.
\end{equation}
The aim of this paper is to find a constructive method of calculating $L$ such that 
\begin{equation*}
 e(t) \rightarrow 0, \qquad t \rightarrow \infty.
\end{equation*}
The main results, Theorem \ref{thm:LocalObserver}, Algorithms \ref{alg:ObserverCvgAlg_new} and \ref{alg:ObserverCvgAlg_Iter}, and Theorem \ref{thm:GlobalCvg}, provide methods of constructing such a gain $L$ by solving a series of convex optimization problems. 

\subsection{A state invariant set}

The property of the non-linear system \eqref{nonlinsys} that is advantageous for observer design is that the energy preserving property \eqref{energy} implies the existence of an {\em invariant set} for the system dynamics. A set $S \subset \mathbb{R}^n$ is said to be invariant for the dynamical system \eqref{nonlinsys} if $x(t_0) \in S$ at time $t_0 \geq 0$ implies that $x(t) \in S$ for every subsequent time $t \geq t_0$. Invariant sets for the class of system \eqref{nonlinsys} can be described in terms of perturbations of the linear part $A$ of the system. In the following, given a matrix $A \in \mathbb{R}^{n \times n}$ and a vector $d \in \mathbb{R}^n$ define a perturbed matrix $A_d \in \mathbb{R}^{n \times n}$ by
\begin{equation*}
A_d x:= Ax + N(x)d + N(d)x, \qquad  x \in \mathbb{R}^n. 
\end{equation*}
Subsequently, we make the following assumption.
\begin{itemize}
\item[(A1)] There exists $d \in \mathbb{R}^n$ such that $A_d \prec 0$. 
\end{itemize}
Clearly, assumption (A1) holds if $A \prec 0$. Furthermore, it is shown in Lemma \ref{lem:ODETrappingSet} that (A1) holds for the class of systems representing finite dimensional approximations of fluid flows.   


\begin{lem}\label{lem:StateTrappingSet}
Suppose that there exist $d \in \mathbb{R}^n$ and $\alpha>0$ such that $A_d + \alpha I \preceq 0$. Then $B_r(d)$ is invariant for \eqref{nonlinsys} for any  
\begin{equation} \label{stateRadius}
r \ge \frac{1}{\alpha} \|Ad + N(d)d\|_2.
\end{equation}
\end{lem}
\begin{proof}
Using the linearity of $N:\mathbb{R}^n \rightarrow \mathbb{R}^{n \times n}$, the system's dynamics can be written
\begin{align*}
\dot x	&= Ax + N(x-d)(x-d) + N(x)d + N(d)x - N(d)d\\
				&= A_d(x-d) + N(x-d)(x-d) + A_d d - N(d)d. 
\end{align*}
For $D(x) :=\frac{1}{2} \|x-d\|_2^2$, the energy preserving property \eqref{energy} implies that
\begin{align*}
							\dot D(x)		 & =		(x-d)\tpose A_d (x-d) + (x-d)\tpose (A d  + N(d)d). \\
													 &\leq	-\alpha \|x-d\|_2^2 + \|x-d\|_2\|A d + N(d)d\|_2.
\end{align*}
Therefore, $\dot D(x) < 0$ whenever $\|x-d\|_2 > r$, with $r$ given by \eqref{stateRadius}. Hence, $B_r(d)$ is invariant for~\eqref{nonlinsys}. 
\end{proof}
If \eqref{nonlinsys} represents a fluid system, an invariant set may be calculated more explicitly, as described in Section \ref{examplesNineSec}. Since $A_d$ is affine in $d$, the condition $ \{ d : A_d  + A_d\tpose \prec 0 \} \neq \emptyset$ can be checked by solving a semidefinite program \cite{Boyd}. 

Ideally, one would like to calculate an invariant ball with the smallest possible radius. However, due to the non-linear dependence of \eqref{stateRadius} upon $(\alpha,d) \in \mathbb{R}\times \mathbb{R}^n$, it is difficult to minimize \eqref{stateRadius} by convex optimization methods.  In order to remove the non-linear dependence upon $d$ from \eqref{stateRadius}, the search can be restricted to vectors $d$ such that \mbox{$N(d)d=0$}. 

We first demonstrate that such vectors always exist. 

\begin{lem}\label{lem:HairyBall}
Suppose that $f : \mathbb{S}^n \to \mathbb{R}^{n+1}$ is a continuous tangent vector field, that is 
\[ x^\top \! f(x)=0, \qquad x \in \mathbb{S}^n,\]
satisfying $f(x) = f(-x)$.  Then $f$ has at least one zero on $\mathbb{S}^n$.
\end{lem}

\begin{proof}
The result is well known when $n$ is odd, in which case the condition $f(x) = f(-x)$ is not required; see \cite{Whittlesey} for a particularly elegant elementary proof.

We therefore consider the case where $n$ is even.  Assume that $f(x) \neq 0$ for all $x \in \mathbb{S}^{n}$, and define the continuous function $g: \mathbb{S}^{n} \to \mathbb{S}^{n}$ as $g(x) := f(x)/\norm{f(x)}$.  Then $g$ does not have a fixed point on $\mathbb{S}^{n}$, since otherwise $x^\top g(x) = x^\top x = 1$ at the fixed point, which is not possible since $g$ is a tangent vector field by construction. 

Since $g$ has no fixed points, its degree is odd \cite[p.~134]{Hatcher} (see \cite[\S2.2]{Hatcher} for a definition of the degree of a continuous map).  On the other hand, since $g(x) = g(-x)$, its degree must be even
\cite[p.~165]{Hatcher}, a contradiction.  Therefore $f$ must have at least one zero.
\end{proof}

The existence of a nonzero $x$ satisfying $N(x)x = 0$ is then guaranteed by setting $f(x) = N(x)x$ in Lemma~\ref{lem:HairyBall}.  Since $x\mapsto N(x)x$ is homogeneous, any such zero also satisfies $N(\alpha x)(\alpha x) = 0$ for all $\alpha \in \mathbb{R}$.  

We henceforward make the following assumption:

\begin{enumerate}[({A}1)]
\addtocounter{enumi}{1}
\item A matrix $Q \in \mathbb{R}^{n \times n}$ is chosen such that $N(d)d =0$ whenever $d \in \ker(Q)$. \label{Assumption:A2}
\end{enumerate}

The advantage of this assumption is that if the search for the centre of an invariant set is conducted over $\ker(Q)$, it can be performed by solving a semidefinite program.  The existence of such a $Q$ is guaranteed, since one can always define $Q := I - zz\tpose$ where $z \neq 0$ is a zero of $N(z)z$ whose existence is guaranteed by Lemma~\ref{lem:HairyBall}.  

We note that identification of such a zero may be difficult in general.   Define symmetric matrices $Q^{(i)}$ such that \mbox{$Q^{(i)}_{jk} := \frac{1}{2}[N(\bs{e}_j)_{ik} + N(\bs{e}_k)_{ij}]$}, for each $i,j,k \in \{1,\ldots,n\}$.  Then 
\begin{equation}\label{eqn:Nquadform}
N(d)d = \left( d\tpose Q^{(1)}d, \ldots, d\tpose Q^{(n)}d \right)\tpose,
\end{equation}
and computing a root of $N(d)d$ amounts to finding a simultaneous root of $n$ quadratic equations in $n$ variables.  See \cite{Barvinok,Grigoriev} for numerical solution methods for such problems.   However, if $\bigcap_{i=1}^n \ker(Q^{(i)}) \supsetneq \{0\}$ it is possible to select nontrivial $Q$ such that
\[
\ker(Q) = \bigcap_{i=1}^n \ker{(Q^{(i)})}.
\]
We demonstrate the application of this method 
to the Lorenz attractor in Section \ref{examplesSec}.  Even if this is not the case, a natural choice for $Q$ may be apparent given the system's underlying structure -- see Section \ref{examplesNineSec}.

%


\begin{prop}\label{prop:smallestStateTrapping}
Suppose that the semidefinite program
\begin{subequations}\label{eqn:SDP}
\begin{align}
\mathrm{minimize}		\quad & s \nonumber \\
\mathrm{subject} \; \mathrm{to} \quad & \left[ \begin{array}{cc} s & (Az)\tpose \\ (\ast) & sI_n \end{array} \right] \succeq 0 \label{eqn:SDP_A}\\
										\quad & tA + (A_z-A) + I \preceq 0 \label{eqn:SDP_B}\\
										\quad & t  \geq 0,\,\,\,Qz = 0  \label{eqn:SDP_C}
\end{align}
\end{subequations}
with variables $s,t \in \mathbb{R}$ and $z \in \mathbb{R}^n$ has optimal solution $(s^\ast,t^\ast,z^\ast)$.  Then $t^\ast>0$ and $B_{s^\ast}(z^\ast / t^\ast)$ is an invariant set for $(x(t))_{t \geq 0}$, with
\begin{align} 	
		 s^\ast = \mathop{\inf_{\substack{\alpha >0\\\,\,\,d \in \mathbb{R}^n}}}  \set{{\alpha}^{-1}{\|Ad + N(d)d\|_2}}{ 
		  A_d + \alpha I \preceq 0, d \in \ker(Q) } < \infty \label{eqn:SDPequiv}.		  
\end{align}
In the particular case $s^* = 0$ and $z^* = 0$, then $B_\gamma(0)$ is an invariant set for $(x(t))_{t \geq 0}$ for any $\gamma\ge 0$.  
If $\ker A \cap \ker Q = \{0\}$, then a minimizer to this SDP is guaranteed to exist if it is feasible.  
\end{prop}

\begin{proof}

We first show that any feasible point of \eqref{eqn:SDP} has $t > 0$.  Assume instead that there is some feasible point with $t = 0$, so that \eqref{eqn:SDP_B} satisfies $(A_z-A) \prec 0$ for some $z$.  Then $z\tpose(A_z-A)z = 2z\tpose N(z)z < 0$, which violates the energy conservation condition \eqref{energy}.

We now consider the case when $s^\ast=0$. Let $(t^\ast,z^\ast)$ be an optimal solution in this case. Then by \eqref{eqn:SDP_A}, \eqref{eqn:SDP_C} we have $z^\ast \in \ker{A}$ and $N(z^\ast)z^\ast  = Q z^\ast  = 0$. Consequently $z^\ast \in \ker (t^\ast A + (A_{z^\ast} - A))$ and \eqref{eqn:SDP_B} implies that $z^\ast =0$. In particular $A_{z^\ast}-A = 0$ and \eqref{eqn:SDP_B} guarantees that $A  + A^\top\prec 0$.  In this case $V(x) = \norm{x}_2^2$ is a Lyapunov function for \eqref{nonlinsys} and a ball of any radius centered at the origin is invariant.  

If $s^* > 0$, then \eqref{eqn:SDP} is equivalent to \eqref{eqn:SDPequiv} after applying a change of variables $\alpha = 1/t$, $d = z /t$ and rewriting \eqref{eqn:SDP_A} as a quadratic constraint via Schur complement.  The inequality \eqref{eqn:SDP_B} is equivalent to $A_d + \alpha I \preceq 0$ (note the identity $t A_{z/t} = tA + (A_z-A)$).  Invariance then follows from Lemma \ref{lem:StateTrappingSet}.

The existence of a minimizer for \eqref{eqn:SDP} can be established by showing that the problem is equivalent to one with compact constraints.  If \eqref{eqn:SDP} is feasible, then any value $s = \bar s$ at any feasible point can be used as an artificial upper bound on $s$.  Since $\ker{A} \cap \ker{Q} =\{0\}$, there exists $\epsilon>0$ such that $\|Az\| \geq \epsilon \|z\|$, for any $z \in \ker{Q}$\footnote{ Suppose for a contradiction that there exists a sequence $z_n \in \ker{Q} \setminus \{0\}$ such that $\|Az_n\| \leq \frac{1}{n} \|z_n\|, n \geq 1$. Let $\xi_n:=z_n / \|z_n\|$. Then $\xi_n \in \mathbb{S}^{n-1}$ and hence there exists a convergent subsequence $\xi_{n_r} \rightarrow \xi \in \mathbb{S}^{n-1} \cap \ker{Q}$. However, 
\[
\|A\xi \| = \lim_{r \rightarrow \infty} \|A \xi_{n_r}\| \leq \lim_{r \rightarrow \infty} \frac{1}{n_r} =  0 
\]
which implies that $\xi \in \ker{A}$, contradicting $\ker{A} \cap \ker{Q} = \{0\}$. \qedhere. 
%
%
%

}. 
Since \eqref{eqn:SDP_A} is equivalent to $\norm{Az} \le s$, we can also add a constraint $\norm{z} \le \bar s/\epsilon =: \bar z$.

Denote as $\xi$ the magnitude of the smallest negative eigenvalue of $\frac{1}{2}(A + A\tpose)$.  Define $\zeta \ge 0$ as 
\[
\zeta := \sup_{\norm{z} \le \bar z}\sigma_{\max}\biggl[I + \frac{1}{2}[(A_z - A) + (A_z - A)\tpose]\biggr].
\]
Then one can also impose an upper bound $t\le \bar t = \zeta/\xi$ without altering the minimum value of \eqref{eqn:SDP}.    Augmenting the constraints in \eqref{eqn:SDP} with $s\le \bar s$, $t\le \bar t$ and $\norm{z} \le \bar z$, so that the feasible set is compact without altering theoptimal value, ensures the existence of a minimizer.

\end{proof}


%

\begin{rem}
Note that the kernel constraint \eqref{eqn:SDP_C} is included in order to ensure that the problem \eqref{eqn:SDP} is solvable as a semidefinite program, and is conservative in the sense that it restricts the search for an invariant set $B_r(d)$ to those with centers satisfying $N(d)d = 0$.  

It is also possible to remove this condition and solve the more general problem
\begin{align} 	
		 s^\ast = \mathop{\inf_{\substack{\alpha >0\\\,\,\,d \in \mathbb{R}^n}}}  \set{{\alpha}^{-1}{\|Ad + N(d)d\|_2}}{ 
		  A_d + \alpha I \preceq 0}		  
\end{align}
directly, i.e.\ to solve the problem \eqref{eqn:SDPequiv} without a kernel constraint.  Assuming that $s^* > 0$, one can make a change of variables $\alpha = 1/t$, $d =  z/t$ and apply the Schur complement to get the equivalent problem
\begin{subequations}\label{eqn:SDPbilin}
\begin{align}
\mathrm{minimize}		\quad & s \nonumber \\
\mathrm{subject} \; \mathrm{to} \quad & \left[ \begin{array}{cc} s & (Az + N(d)z)\tpose \\ (\ast) & sI_n \end{array} \right] \succeq 0 \\
										\quad & tA + (A_z-A) + I \preceq 0 \\
										\quad & t  \geq 0, \,\, z = td.
\end{align}
\end{subequations}
Noting that $z = td$ is equivalent to the pair of constraints
$\mathrm{diag}(z) \preceq \mathrm{diag}(td) \preceq \mathrm{diag}(z)$, the constraints in optimization problem \eqref{eqn:SDPbilin} constitute a set of bilinear matrix inequalities (BMIs).  Although methods  for solving optimization problems of this type are available \cite{PENNON}, there is generally no guarantee that a solution will be globally optimal.  We therefore follow the somewhat more conservative method of 
Proposition~\ref{prop:smallestStateTrapping}.
\end{rem}

\subsection{Locally stable observers}

If it is possible to calculate an invariant set for the state, Theorem \ref{thm:LocalObserver} of this section provides a strategy for constructing a locally convergent observer. First, it will be useful to derive an explicit expression for the norm of the non-linear term $N$. 

\begin{lem} \label{lem:Nnorm}
Suppose that $N$ is given by \eqref{eqn:Nquadform}. Define matrices $\tilde Q^{(k)} = (\tilde q_{ij}^{(k)})_{i,j=1}^n$ by $\tilde q_{ij}^{(k)}:= q_{jk}^{(i)}$. Then 
\[
\|N\|_2 = \sigma_{\mathrm{max}}( \Theta )^\frac{1}{2}, 
\]
where $\Theta$ is a matrix whose $(i,j)^{th}$ entry is $ \Theta_{i,j} := \innerprod{\tilde Q^{(i)}}{\tilde Q^{(j)}}$, and $\innerprod{A}{B} =\mathop{tr}(A\tpose B)$ is the standard (Frobenius) inner product. 
\end{lem}
\begin{proof}
It can be shown from (\ref{eqn:Nquadform}) that $N(x) = \sum_{i=1}^n x_i \tilde Q^{(i)}$, for $x = (x_i)_{i=1}^n \in \mathbb{R}^n$. The adjoint $N^\ast : \mathbb{R}^{n \times n} \rightarrow \mathbb{R}^n$ is given by 
\[
N^\ast(Q) = \left( \langle{Q,\tilde Q^{(i)}}\rangle \right)_{i=1}^n, \qquad Q \in \mathbb{R}^{n \times n},
\]
and hence, $N^\ast N = \Theta$.
\end{proof}


\begin{thm}[Local Observer Convergence] \label{thm:LocalObserver}
Suppose that $B_r(d)$ is invariant for $(x(t))_{t \geq 0}$ and $Y \subset \mathbb{R}^n$ satisfies $B_r(d) \subseteq Y$. Suppose that there exist $(\alpha_i)_{i=1}^3>0, P \in \mathcal{S}^n$ and $R \in \mathbb{R}^{n \times p}$ such that
\begin{equation}
\alpha_1 I  \preceq  P  \preceq  \alpha_2 I \label{Pbnd}
\end{equation}
and
\begin{equation}\label{dynSucc00}
PA_{y} + A^\top_{y}P + RC + C^\top R^\top  \preceq  -\alpha_3 I, \qquad y \in Y. 
\end{equation}
If $(\hat x(t))_{t \geq0}$ has dynamics \eqref{LuOb} for $L:=P^{-1}R$ and 
\begin{equation} \label{goodGuess}
\|e_0\|_2 = \|x_0 - \hat x_0\|_2 < \frac{\alpha_3}{2 \gamma \alpha_2}\sqrt{ \frac{\alpha_1}{\alpha_2}},
\end{equation}
then $e(t)\rightarrow 0, t \rightarrow \infty$. The constant $\gamma:=\sigma_{\mathrm{max}}(\Theta)^\frac{1}{2}$ is defined in Lemma \ref{lem:Nnorm}. 
\end{thm}
\begin{proof}
Using 
\begin{align}
N(x)x-N(\hat x)\hat x &= N(x-\hat x)x - N(\hat x)\hat x +N(\hat x)x \nonumber\\
											&= N(x-\hat x)x + N(\hat x)(x-\hat x)  \nonumber\\
											&=N(e)x + N(x)e - N(e)e, \label{eqn:nonlinError_v1}
\end{align}														
the error dynamics (\ref{Err}) can be rewritten
\begin{equation*}
\dot e = (A_{x(t)} + LC)e -N(e)e, \qquad t \geq 0. 
\end{equation*}
Since the underlying state dynamics $(x(t))_{t \geq 0}$ are unaffected by $(e(t))_{t \geq 0}$ it is possible to consider the time varying linear operator
$\mathcal{A}(t):= A_{x(t)} + LC$ independently of the error dynamics. Hence, if $V(e):=e^T Pe$, 
\begin{align}
\dot V(e) & = e^\top \left( P\mathcal{A}(t) + \mathcal{A}(t)^\top P\right) e - 2e^\top PN(e)e \nonumber \\
(\text{by (\ref{Pbnd}), (\ref{dynSucc00})}) & \leq  - \alpha_3 \|e\|_2 + 2\alpha_2 \|e\|_2^2 \|N(e)\|_2 \nonumber \\
(\text{by Lemma \ref{lem:Nnorm}}) & \leq (-\alpha_3 + 2 \gamma \alpha_2 \|e\|_2 )\|e\|_2^2. \label{deriv:eqn}
\end{align}
Now let $\epsilon \in (0,\alpha_3)$ be such that $\|e_0\|_2 = \frac{\alpha_3-\epsilon}{2 \gamma \alpha_2}\sqrt{ \frac{\alpha_1}{\alpha_2}}$ and define $\rho:=(\alpha_3-\epsilon)/2\gamma \alpha_2$. Then \eqref{Pbnd} and \eqref{deriv:eqn} imply that,
\[
V(e) \leq \alpha_1 \rho^2 \Longrightarrow \|e\|_2 \leq \rho \Longrightarrow \dot V(e) \leq -\epsilon \|e\|_2^2.
\]
By \eqref{Pbnd} and \eqref{goodGuess}, $V(e_0) \leq \alpha_2 \|e_0\|_2^2 =\alpha_1 \rho^2$. Hence, $(V(e(t)))_{t \geq 0}$ is decreasing and 
\[
\|e(t)\|_2 \leq \sqrt{ \frac{V(e(t))}{\alpha_1}} \leq \sqrt{ \frac{ V(e_0)}{\alpha_1}}   e^{-\frac{\epsilon t}{2 \alpha_2}} \leq \rho e^{-\frac{\epsilon t}{2 \alpha_2}} \rightarrow 0, \qquad t \rightarrow \infty.
\]
\end{proof}
\begin{rem} \label{neccsuff:rem}
A simple, but instructive, necessary condition for \eqref{Pbnd}, \eqref{dynSucc00} to hold is that the pair $(A_d,C)$ is detectable. In other words, the output map must at least be compatible with the linear system generated by the perturbed matrix $A_d$. Furthermore, if $B_r(d)$ is invariant for $(x(t))_{t \geq 0}$ with $A_d + \alpha I \preceq 0$ and $r= \alpha^{-1} \|Ad + N(d)d\|_2$, then 
\[
\|d\|_2^2 \leq -\frac{1}{\alpha} d^\top A_d d = -\frac{1}{\alpha} d^\top (Ad + N(d)d) \leq \frac{1}{\alpha} \|d\|_2 \|Ad + N(d)d\|_2 = r \|d\|_2.
\]
Hence, $0 \in B_r(d)$. Therefore, in the case that the state invariant set is calculated by Proposition \ref{prop:smallestStateTrapping}, detectability of $(A,C)$  is also necessary for \eqref{Pbnd}, \eqref{dynSucc00} to hold.

A sufficient condition for local convergence can be formulated involving only the matrices $A_d$ and $C$. If there exists $\alpha>0$ and  $0 \prec P \prec \alpha( 4 \gamma r)^{-1}$ such that 
\[
PA_d + A_d^\top P + RC + C^\top R^\top \preceq -\alpha I,
\]
then it is not difficult to show that \eqref{dynSucc00} holds for $Y = B_r(d)$, implying that the observer is locally convergent. 
\end{rem}

We discuss a method for reformulating the semi-infinite LMI constraint \eqref{dynSucc00} as a finite-dimensional LMI in Section \eqref{sec:RobustLMI}.

\subsection{An observer invariant set}

If a locally convergent observer can be constructed by Theorem \ref{thm:LocalObserver}, it is natural to ask whether it is possible to extend the set of initial states for which the observer converges. Since it is known that the system state $(x(t))_{t \geq 0}$ has an invariant set $S$, say, it is desirable for the observer $(\hat x(t))_{t \geq 0}$ to itself possess an invariant set which contains $S$.




The following two results provide a method for calculating an invariant set for the observer dynamics.  The first of these characterizes the trapping set for the observer error dynamics, and parallels the results of Lemma~\ref{lem:StateTrappingSet}.

\begin{lem} \label{lem:ObserverTrappingSet}
Suppose that $B_r(d)$ is invariant for \eqref{nonlinsys} and that $x_0 \in B_r(d)$. Suppose that there exist $\hat d \in \mathbb{R}^n$ and $\alpha>0$ such that  $A_{\hat d} + LC +\alpha I \preceq 0$, then $B_{\hat r}(\hat d)$ is invariant for \eqref{LuOb} for any 
\begin{equation} \label{obsRadius}
\hat r \ge \frac{1}{\alpha} \sup_{v \in B_r(d-\hat d) } \left\| LCv - A\hat d - N(\hat d)\hat d  \right\|_2.
\end{equation}
\end{lem}
\begin{proof}
Suppose that $\hat d \in \mathbb{R}^n$ is such that $A_{\hat d} + LC + \alpha I \preceq 0$. Then, 
\begin{align*}
\dot{ \hat x} = (A_{\hat d}+LC)(\hat x - 	& \hat d) + N(\hat x - \hat d)(\hat x - \hat d)\\
																					& - (LC(x-\hat d)- A\hat d-N(\hat d)\hat d). 
\end{align*}
For $D(\hat x):= \frac{1}{2} \|\hat x - \hat d\|_2^2$, the energy preserving property \eqref{energy} implies that 
\begin{align*}
\dot D(\hat x)	&= (\hat x - \hat d)\tpose (A_{\hat d} + LC )(\hat x-\hat d) \\
& \qquad \qquad - (\hat x -\hat d)\tpose ( LC(x-\hat d)- A\hat d - N(\hat d)\hat d )\\
								&\le -\alpha \|\hat x - \hat d\|_2^2 \\
								& \qquad \qquad - (\hat x - \hat d)\tpose ( LC(x- \hat d) - A\hat d - N(\hat d)\hat d ).
\end{align*}
By assumption,  $x(t) \in B_r(d)$, for any time $t \geq 0$, so that
\begin{equation*}
x(t)-\hat d \in B_r(d-\hat d), \qquad t \geq 0.
\end{equation*}
Therefore $\dot D(\hat x) < 0$ whenever $\norm{\hat x - \hat d} > \hat r$, with $\hat r$ given by \eqref{obsRadius}.    Hence, $B_{\hat r}(\hat d)$ is invariant for~\eqref{LuOb}.
\end{proof}


The next result provides a method for computing an invariant set for the observer dynamics given an observer gain $L$.  As in case for the state invariant set,  the non-linear dependence of \eqref{obsRadius} upon $(\alpha,\hat d) \in \mathbb{R}\times \mathbb{R}^n$ makes global minimization difficult.  We therefore remove the non-linear dependence upon $\hat d$ from \eqref{obsRadius} by restricting the search to vectors $\hat d$ such that $N(\hat d)\hat d=0$.  In the following, the assumption is made that the observer gain $L$ is such that $A+LC$ is stable. Note that, by Remark \ref{neccsuff:rem}, this is a necessary condition even for local convergence. The following result parallels Proposition \ref{prop:smallestStateTrapping}.


\begin{prop}\label{prop:smallestObsTrapping}
Suppose that $B_r(d)$ is an invariant set for \eqref{nonlinsys} with $r > 0$, $A+LC$ stable, and that the semidefinite program
\begin{subequations}\label{eqn:SDPobs}
\begin{align}
\mathrm{minimize}		\quad & s \nonumber \\
\mathrm{subject} \; \mathrm{to} \quad & 
\left[ \begin{array}{cc} s-rt\|LC\|_2 & (LC(td-z) - Az)\tpose \\ 	
(\ast) & (s-rt\|LC\|_2)I_n \end{array} \right]  \succeq 0 	\label{eqn:SDPobs_A}\\
\quad & t(A+LC) + (A_{z}-A) + 	I  \preceq 0 \label{eqn:SDPobs_B}\\
\quad & t  \geq 0,\,\,\,Qz = 0  \label{eqn:SDPobs_C}
\end{align}
\end{subequations}

with variables $s,t \in \mathbb{R}$ and $z \in \mathbb{R}^n$ has optimal solution $(s^\ast,t^\ast,z^\ast)$.  Then $(s^\ast,t^\ast)>0$ and $B_{s^\ast}(z^\ast / t^\ast)$ is an invariant set for $(\hat x(t))_{t \geq 0}$, with
\begin{align} 	
		 s^\ast \ge \mathop{\inf_{\substack{\alpha >0\\\,\,\,\hat d \in \mathbb{R}^n}}}  \set{{\alpha}^{-1}
		 \!\!\!\!\!\!\!
		 \sup_{v \in B_r(d-\hat d) } \norm{LCv - A\hat d - N(\hat d)\hat d }_2
		 }{ 
		  A_{\hat d} + LC +\alpha I \preceq 0, \hat d \in \ker(Q) }.		  
\end{align}
A minimizer to this SDP is guaranteed to exist if it is feasible.  
\end{prop}

\begin{proof}
If $LC=0$, then $A$ is stable and the result follows from Proposition \ref{prop:smallestStateTrapping}. 

Now suppose that $LC \neq 0$. Using the same argument as in the proof of Proposition~\ref{prop:smallestStateTrapping}, one can show that any feasible point of \eqref{eqn:SDPobs} must satisfy $t > 0$.  The constraint \eqref{eqn:SDPobs_A} then requires \mbox{$s \ge rt\norm{LC}_2 > 0$}.

Noting that $Q\hat d = 0 \Rightarrow N(\hat d)\hat d = 0$ by assumption, the smallest invariant set radius $\bar r$ satisfying the inequality \eqref{obsRadius} for a given $\alpha$ simplifies to
\begin{align}
\bar r &= \sup_{v \in B_r(d-\hat d) }\alpha^{-1} \norm{LCv - A\hat d}_2 \nonumber\\
	   &= \sup_{v \in B_r(0)}\alpha^{-1} \|LCv + LC(d-\hat d) - A\hat d \|_2 \nonumber\\
	   &\le r\alpha^{-1}\|LC\|_2 + \alpha^{-1}\|LC(d-\hat d) - A\hat d \|_2 \label{obsRadius_CSbound}.
\end{align}

Upper bounding \eqref{obsRadius_CSbound} by $s$ and substituting $t = 1/\alpha$ and $z = \hat d/\alpha$, results in 
\[
\bar r \leq rt \|LC\|_2 + \|LC t d- (A+LC)z\|_2	  \leq s.
\]
and applying a Schur complement identity produces the equivalent linear matrix inequality~\eqref{eqn:SDPobs_A}.    One may likewise confirm that the inequality $A_{\hat d} + LC +\alpha I \preceq 0$ is equivalent to~\eqref{eqn:SDPobs_B}.  Invariance then follows from Lemma \ref{lem:ObserverTrappingSet}.

To establish the existence of a minimizer, we first show that the SDP \eqref{eqn:SDPobs} is equivalent to one with compact constraints.  Any value $s = \bar s$ at any feasible point can be used as an upper bound on $s$, which allows an additional constraint $t \le \bar t := \bar s/(r\norm{LC}_2)$ to be imposed as a necessary condition for~\eqref{eqn:SDPobs_A}.
A further necessary condition for \eqref{eqn:SDPobs_A} is then
\[
\norm{(A+LC)z}_2 \le \bar s + \norm{LCd}\bar t.
\]
Noting that $\ker(A+LC) = \{0\}$ since $(A+LC)$ is assumed stable, the remainder of the proof proceeds as in the proof of Proposition~\ref{prop:smallestStateTrapping}.

\end{proof}

\begin{rem}
Note that the proofs of Propositions \ref{prop:smallestStateTrapping} and \ref{prop:smallestObsTrapping} are similar, but that the result of Proposition \ref{prop:smallestStateTrapping} produces a \emph{tight} bound on the invariant set radius for the state dynamics \eqref{nonlinsys}, whereas the result of Proposition \ref{prop:smallestObsTrapping} is conservative due to the application of the triangle inequality in \eqref{obsRadius_CSbound}.
\end{rem}


%

\subsection{Globally stable observers}

To study observer convergence, it is useful to rewrite the nonlinear part of the observer error dynamics \eqref{Err}. Starting from \eqref{eqn:nonlinError_v1}, for $x, \hat x \in \mathbb{R}^n$ the nonlinear error term is 
\begin{align}
N(x)x-N(\hat x)\hat x &=N(e)x + N(x)e - N(e)e	\nonumber \\
					 &=N(e)(x-e/2) + N(x-e/2)e	 \nonumber	\\
					 &= N(e)((x+\hat x)/2) + N((x+\hat x)/2)(e) \nonumber	\\
					 &= (A_{\frac{x+ \hat x}{2}} - A)e. \label{errGlobal}
\end{align}
Hence, the error dynamics \eqref{Err} can be written
\begin{equation} \label{ErrConcise}
\dot e = \left( A_{\frac{x+ \hat x}{2}} +LC  \right)e.
\end{equation}


The observer error dynamics can therefore be considered as a linear time varying system, and the problem of observer design is to find a gain $L$ which stabilizes \eqref{ErrConcise}.    Since both the state and error dynamics can be contained inside separate invariant sets $B_r(d)$ and $B_{\hat r}(\hat d)$ respectively, our objective is to identify a gain $L$ that stabilizes \eqref{ErrConcise} under the assumption that 
\begin{equation} \label{stateObsAve}
\left(x(t) + \hat x(t) \right) \in B_r(d) \oplus B_{\hat r}(\hat d), \qquad t \geq 0.
\end{equation}
The central difficulty is of course that the estimation error trapping set $B_{\hat r}(\hat d)$ is itself determined by the observer gain $L$.  We therefore propose a two-phase strategy, which we characterize formally in Algorithm \ref{alg:ObserverCvgAlg_new}.  

Our general approach is first to identify a trapping set $B_{r}(d)$ for the state dynamics \eqref{nonlinsys} using the method of Proposition \ref{prop:smallestStateTrapping}.  We then select some set $Y$ such that $B_r(d) \subsetneq Y$, and compute a gain $L$ such that $A_y+LC$ is stable for all $y \in Y$.  Using this gain, one can compute an observer invariant set $B_{\hat r}(\hat d)$ using the results of Proposition \ref{prop:smallestObsTrapping}.  If such a set exists and $(B_r(d) \oplus B_{\hat r}(\hat d))/2 \subseteq Y$, then \eqref{ErrConcise} is stable and $e(t) \rightarrow 0$.  

\begin{alg}[Observer design] \label{alg:ObserverCvgAlg_new}
\mbox{}
\begin{enumerate}
\item\label{alg:Obs_A} Use Proposition \ref{prop:smallestStateTrapping} to select $d \in \mathbb{R}^n, r>0$ such that $B_r(d)$ is invariant for \eqref{nonlinsys}. 
\item\label{alg:Obs_B} Select $(\alpha_1, \alpha_2) \geq 0$ and $Y \subset \mathbb{R}^n$ such that $B_r(d) \subsetneq Y$. Compute a positive definite $P \in \mathcal{S}^n$ and $R \in \mathbb{R}^{n \times p}$ such that:  
\begin{eqnarray}
P -\alpha_1 I & \succ & 0; \label{invP}\\
\left[ \begin{array}{cc} \alpha_2 I_n & RC \\ (\ast) & \alpha_2 I_p \end{array} \right] & \succeq &0 ; \label{RC}\\
PA_y + A\tpose_yP + RC + C\tpose R\tpose &\prec & 0, \qquad y\in Y.\label{dynSucc0}
\end{eqnarray}
Define $L := P^{-1} R$.
\item\label{alg:Obs_C} Use Proposition \ref{prop:smallestObsTrapping} to select $\hat d \in \mathbb{R}^n, \hat r>0$ such that $B_{\hat r}(\hat d)$ is invariant for \eqref{LuOb}. 

\item\label{alg:Obs_D} If 
$\frac{1}{2}\left[B_r(d) \oplus B_{\hat r}(\hat d)\right] \subset Y,$
then \eqref{ErrConcise} is stable and $e(t) \rightarrow 0$, whenever $\hat x_0 \in B_{\hat r}(\hat d)$.  

\end{enumerate}
\end{alg}


\begin{rem}

The tuning parameters $(\alpha_1,\alpha_2)$ appearing in Step \ref{alg:Obs_B}.~of Algorithm \ref{alg:ObserverCvgAlg_new} are included to provide control over $\|LC\|_2$. Recalling \eqref{obsRadius_CSbound} in the proof of Proposition \ref{prop:smallestObsTrapping}, $\|LC\|_2$ influences the radius of the observer invariant set calculated in Step \ref{alg:Obs_C}.   Minimizing the size of the this set is useful in helping to ensure that the set inclusion in Step \ref{alg:Obs_D}. is satisfied. 

To ensure that Proposition \ref{prop:smallestObsTrapping} can be used to construct an invariant set for the observer in Algorithm \ref{alg:ObserverCvgAlg_new}, Step \ref{alg:Obs_C}.~one must of course first verify that 
\[
\set{\alpha}{\alpha > 0, \exists \hat d \in \mathbb{R}^n,  A_{\hat d} + LC + \alpha I \prec 0} \neq \emptyset.
\]

\end{rem}

The fact that $R$ and $P$ are searched for simultaneously in \eqref{invP}--\eqref{dynSucc0} may mean that $\|LC\|_2$ is suboptimal, and consequently that Step \ref{alg:Obs_D}.~of Algorithm \ref{alg:ObserverCvgAlg_new} does not hold. In this situation, we propose the following iterative search for a globally convergent observer.

\begin{alg} \label{alg:ObserverCvgAlg_Iter}
$\mathrm{Initialization:}$ Suppose that Steps  \ref{alg:Obs_A}.~--~\ref{alg:Obs_C}. of  Algorithm \ref{alg:ObserverCvgAlg_new}  have been completed to provide:
\begin{enumerate}
\item $d \in \mathbb{R}^n, r>0$ such that $B_r(d)$ is invariant for \eqref{nonlinsys};
\item $P \in \mathcal{S}^n, R \in \mathbb{R}^{n \times p}$ and $Y \subset \mathbb{R}^n$ such that Step \ref{alg:Obs_B}. of Algorithm \ref{alg:ObserverCvgAlg_new} holds;
\item $\hat d \in \mathbb{R}^n, \hat r>0$ such that $B_{\hat r}(\hat d)$ is invariant for \eqref{LuOb}.
\end{enumerate}
Define $P_0:=P, L_0:=P^{-1}R, \hat d_0 :=\hat d, \hat r_0 := \hat r, \alpha_0 :=0$ and $\beta_0 := 0$. 

$\mathrm{Iteration:}$ repeat until $\frac{1}{2}\left[B_r(d) \oplus B_{\hat r_k}(\hat d_k)\right] \subset Y$: 
\begin{enumerate}
\item\label{alg:Iter_A} Suppose that $\alpha^\ast >0$ and $P^\ast \in \mathcal{S}^n$ are an optimal solution to the semidefinite program
\begin{align*}
\mathrm{maximize}		\quad & \alpha \nonumber \\
\mathrm{subject} \; \mathrm{to} \quad & 
P \succeq \alpha I\\
\quad & P(A_y+L_k C) + (A_y + L_k C)^\top P \preceq 0, \qquad y \in Y.
\end{align*}
Let $P_{k+1}:=P^\ast, \alpha_{k+1}:=\alpha^\ast$. 

\item\label{alg:Iter_B} Suppose that $\beta^\ast>0$ and $L^\ast \in \mathbb{R}^{n \times p}$ are an optimal solution to the semidefinite program
\begin{align*}
\mathrm{minimize} \quad & \beta \nonumber \\
\mathrm{subject} \; \mathrm{to}  \quad &
\left[ \begin{array}{cc} \beta I_n & LC \\ (\ast) & \beta I_p \end{array} \right]  \succeq  0 \\
& P_{k+1}(A_y + LC) + (A_y + LC)^\top P_{k+1} \preceq 0, \qquad y \in Y.
\end{align*}
Let $L_{k+1}:=L^\ast, \beta_{k+1} := \beta^\ast$.

\item\label{alg:Iter_C} Apply Proposition  \ref{prop:smallestObsTrapping} with $L=L_{k+1}$ to select $\hat d_{k+1} \in \mathbb{R}^n, \hat r_{k+1}>0$ such that $B_{\hat r_{k+1}}(\hat d_{k+1})$ is invariant for \eqref{LuOb}. 

\item\label{alg:Iter_D} If $\frac{1}{2} \left[ B_r(d) \oplus B_{\hat r_{k+1}}(\hat d_{k+1}) \right] \subset Y$, then $e(t) \rightarrow 0$ for any $\hat x_0 \in B_{\hat r_{k+1}}(\hat d_{k+1})$. 

\end{enumerate}

\end{alg}

\begin{rem} \label{rem:IterativeAlg}
At each stage of the iterative proceedure, $(P_k, \lambda_{\mathrm{min}}(P_k) )$ is feasible for the SDP in Step \ref{alg:Iter_A}., while $(L_k,\|L_k C\|_2)$ is feasible for the SDP in Step \ref{alg:Iter_B}. Furthermore, for each $k \geq 1$, we have the bound $\|L_k C\|_2 \leq \beta_k / \alpha_k$. 
\end{rem}

An alternative to the iterative method of Algorithm \ref{alg:ObserverCvgAlg_Iter} is to search for $P$ over a particular subset of $\mathcal{S}^n$, defined in terms of the non-linearity $N$. If this subset is well chosen, it is possible to remove the need to find $B_{\hat r}(\hat d)$. Define
\begin{equation*}
\mathcal{S}_{\!N}^n := \{ P \in \mathcal{S}^n : e\tpose P N(e)e = 0,  \; \text{for each} \; e \in \mathbb{R}^n \}. 
\end{equation*} 
Since the energy preserving property \eqref{energy} holds, it is the case that $\mathcal{S}_{\! N}^n \neq \emptyset$. Notice also that, since $e\tpose PN(e)e$ is linear in $P$, it is easy to calculate $\mathcal{S}_{\! N}^n$ for a given non-linearity $N$. The following result provides conditions for global observer convergence. 

\begin{thm} \label{thm:GlobalCvg}
Suppose that $B_r(d)$, calculated by Proposition \ref{prop:smallestStateTrapping}, is invariant for \eqref{nonlinsys} and let $x_0 \in B_r(d)$. Pick $Y \subseteq \mathbb{R}^n$ such that $B_r(d) \subset Y$ and suppose that there exists a positive definite $P \in \mathcal{S}_{\! N}^n$ and $R \in \mathbb{R}^{n \times p}$ such that 
\begin{equation} \label{GlobalPrec}
PA_y + A_y^\top P + RC + C^\top R^\top \prec 0, \qquad y \in Y.
\end{equation}
Then if $L:=P^{-1}R$, the observer $(\hat x(t))_{t \geq 0}$ defined by \eqref{LuOb} satisfies $e(t) \rightarrow 0$, $t\rightarrow \infty$, for any initial condition $\hat x_0 \in \mathbb{R}^n$. 
\end{thm}
\begin{proof}
Let $V(e) = e\tpose P e$. Then by \eqref{errGlobal} and \eqref{ErrConcise},
\begin{align*}
					\nabla V \cdot \dot e		& =			e\tpose P ( A_{\frac{x+\hat x}{2}} + LC)e  \\
																	& = 		e\tpose P ( A + LC)e \\
																	&	\qquad \qquad+ e\tpose P(N(x)e + N(e)x - N(e)e) \\
(P \in \mathbb{S}_N^n)						& = 		e\tpose P (A_x + LC)e. 										 \\
\text{(by \eqref{GlobalPrec})}		& < 	0, \qquad e \in \mathbb{R}^n.
\end{align*}
Hence, $e(t) \rightarrow 0, t \rightarrow \infty$ and since $\hat x$ does not appear in the expression for $\nabla V \cdot \dot e$, the observer error converges to zero for any initial condition $\hat x_0 \in \mathbb{R}^n$. 
\end{proof}

\subsection{Modeling of Robust LMI Conditions}\label{sec:RobustLMI}

In order to apply the results of Theorem \ref{thm:LocalObserver}, Algorithms \ref{alg:ObserverCvgAlg_new} and \ref{alg:ObserverCvgAlg_Iter}, or Theorem \ref{thm:GlobalCvg}, it is necessary to construct matrices $P$ and $R$ such that the semi-infinite matrix inequality 
\begin{equation} \label{robustLMI:eqn}
PA_y + A_y^\top P + RC + C^\top R^\top \preceq -\alpha I, \qquad y \in Y,
\end{equation}
is satisfied for some compact set $Y \subseteq  \mathbb{R}^n$ and for some $\alpha > 0$.  We next comment on methods for modeling such a constraint as a finite-dimensional LMI.

Suppose that $Y = B_r(d)$, so that \eqref{robustLMI:eqn} can be rewritten as 
\[
PA_d + A_d^\top P + RC + C^\top R^\top  + P(A_\delta - A) + (A_\delta - A)^\top P\preceq -\alpha I, \qquad  \norm{\delta}_2\le r.
\]
Define $\xi := (P,R,\alpha)$, a matrix $F^{(0)}(\xi)\in \mathcal{S}^n$ as
\[
F^{(0)}(\xi) := -\left(PA_d + A_d^\top P + RC + C^\top R^\top + \alpha I\right), 
\]
and matrices $F^{(i)}(\xi)\in \mathcal{S}^n$ for $i \in \{1,\dots, n\}$ such that
\[
-P(A_\delta - A) - (A_\delta - A)^\top P =: \sum_{i=1}^n\delta_iF^{(i)}(\xi).
\]  
The robust LMI condition \eqref{robustLMI:eqn} can be rewritten in this notation as
\[
F^{(0)}(\xi) + \sum_{i=1}^n\delta_iF^{(i)}(\xi) \succeq 0, \qquad \norm{\delta}_2 \le r,
\]
where each of the matrices $F^{(i)}(\xi)$ is linear in $\xi$.  We can then exploit a result from robust semidefinite programming to establish a sufficient condition for satisfaction of \eqref{robustLMI:eqn}.

\begin{prop}[{\cite[Thm.\ 2.1]{BenTal}}]  If $Y = B_r(d)$, then the robust LMI \eqref{robustLMI:eqn} is satisfied if there exists $\alpha > 0$, $Q \in \mathcal{S}^n$ and $S\in \mathcal{S}^n$ such that $(S + Q) \preceq 2F^{(0)}(\xi)$ and 
\[
\begin{bmatrix}
S & r F^{(1)}(\xi) & \cdots & r F^{(n)}(\xi) \\
\rho F^{(1)}(\xi) & Q \\
\vdots &&\ddots \\
\rho F^{(n)}(\xi) &&& Q
\end{bmatrix}
\succeq 0. 
\]
\end{prop}

In the more general case that $Y \subseteq \text{conv} \{ y_i, i=1,\ldots,M \}$, one can of course also guarantee satisfaction of the constraint \eqref{robustLMI:eqn} by ensuring its satisfaction  at every vertex $y_i \in Y$.

\section{Examples} \label{examplesSec}

We give two examples of observer design for finite dimensional systems related to fluid flows; the Lorenz attractor and a low order model for shear flow between two parallel plates.

\subsection{Lorenz Attractor}

The Lorenz attractor \cite{Lorenz} is a dynamical system in $\mathbb{R}^3$, which is a simplified model of fluid convection in two spatial dimensions. We consider the classical Lorenz dynamics which can be written in the form \eqref{nonlinsys} for 
\begin{align*}
A 		&	:= 	\left( \begin{array}{ccc} -10 & 10 & 0 \\ 28 & -1 & 0 \\ 0 & 0 & -8/3 \end{array} \right),  \\
N(x)	&	:=	\left( \begin{array}{ccc} 0 & 0 & 0 \\ 0 & 0 & -x_1 \\ 0 & x_1 & 0 \end{array} \right), \quad x \in \mathbb{R}^3.
\end{align*}
Note that non-linearity $N$ satisfies the energy preserving property \eqref{energy}. 
Suppose that it is possible to observe only the second state: 
\begin{equation*}
y = Cx, \qquad C := \left(\begin{array}{ccc} 0 & 1 & 0 \end{array} \right). 
\end{equation*} 
A convergent observer can be constructed using the iterative method of Algorithm \ref{alg:ObserverCvgAlg_Iter}. Proposition \ref{prop:smallestStateTrapping} implies that $B_{r}(d)$ is invarient for $(x(t))_{t \geq 0}$ with 
\[
r:=100.7,\qquad  d:= \left( \begin{array}{ccc} 0 & 0 & 37.5 \end{array} \right)^\top. 
\]
Select a conservative bounding set $Y = B_{r_0}(d_0)$ with $r_0:=1200$ and $d_0:=d$. After five iterations, Algorithm \ref{alg:ObserverCvgAlg_Iter} provides matrices\footnote{The additional condition $P \preceq 10^3 I$ was imposed to improve convergence.}
\[
L_5 = \left( \begin{array}{ccc} -10.0 & -13.3 & 0 \end{array} \right), \qquad P_5 = \text{diag}\left( \begin{array}{ccc} 1000.0 & -0.1 & -0.1 \end{array} \right)
\]
which satisfy $P_5(A_y+L_5 C) + (A_y + L_5 C)^\top P_5 \prec 0, y \in Y$. The set $B_{\hat r_5}(\hat d_5)$ is invariant for the observer dynamics \eqref{LuOb} with
\[
\hat r_5 = 1282.6, \qquad \hat d_5 = \left( \begin{array}{ccc} 0 & 0 & 9.2 \end{array} \right)^\top 
\]
and it can be easily verified that $\frac{1}{2} \left[ B_r(d) \oplus B_{\hat r_5} (\hat d_5) \right] \subset Y$. Hence, Algorithm \ref{alg:ObserverCvgAlg_Iter} implies that $\hat x(t) \rightarrow x(t), t \rightarrow \infty$ whenever 
\[
(x_0,\hat x_0) \in B_r(d) \times B_{\hat r_5}(\hat d_5). 
\]

Alternatively, Theorem \ref{thm:GlobalCvg} can be used to construct a globally convergent observer. For the Lorenz attractor, we have 
\begin{equation*}
\mathcal{S}_N^n = \text{span} \left\{ \text{diag}\left(\begin{array}{ccc} 1 & 0 & 0 \end{array} \right), \text{diag}\left( \begin{array}{ccc}  0 & 1 & 1 \end{array} \right) \right\}.
\end{equation*}  	
and it is interesting to note that the matrix $P_5$ constructed by Algorithm \ref{alg:ObserverCvgAlg_Iter} is an element of $\mathcal{S}_N^n$. 
Applying Algorithm \ref{alg:ObserverCvgAlg_new}, Step \ref{alg:Obs_B}.~with the restrictions $P \in \mathcal{S}^n_N, P \prec 10^3 I$ provides matrices 
\begin{equation} \label{eqn:Lgain}
L= \left( \begin{array}{ccc}   -9.6& -704.4 & 0 \end{array} \right)^\top, \qquad P =	\text{diag}\left(\begin{array}{ccc} 132.4 &0.8 &0.8 \end{array} \right) \in \mathcal{S}_N^n
\end{equation}
which satisfy $P(A_{y} +LC) + (A_{y} +LC)^\top P  \prec 0, y \in B_r(d)$. Hence, Theorem \ref{thm:GlobalCvg} implies that the resulting observer is globally convergent for any initial value $\hat x_0 \in \mathbb{R}^3$. An example of the performance of the two globally convergent observers is shown in Figure \ref{LorzFig}. 

\begin{figure*}[h!]
\begin{center} 
\includegraphics[scale=0.53,trim = 1.75cm 2cm 1.3cm 1.5cm]{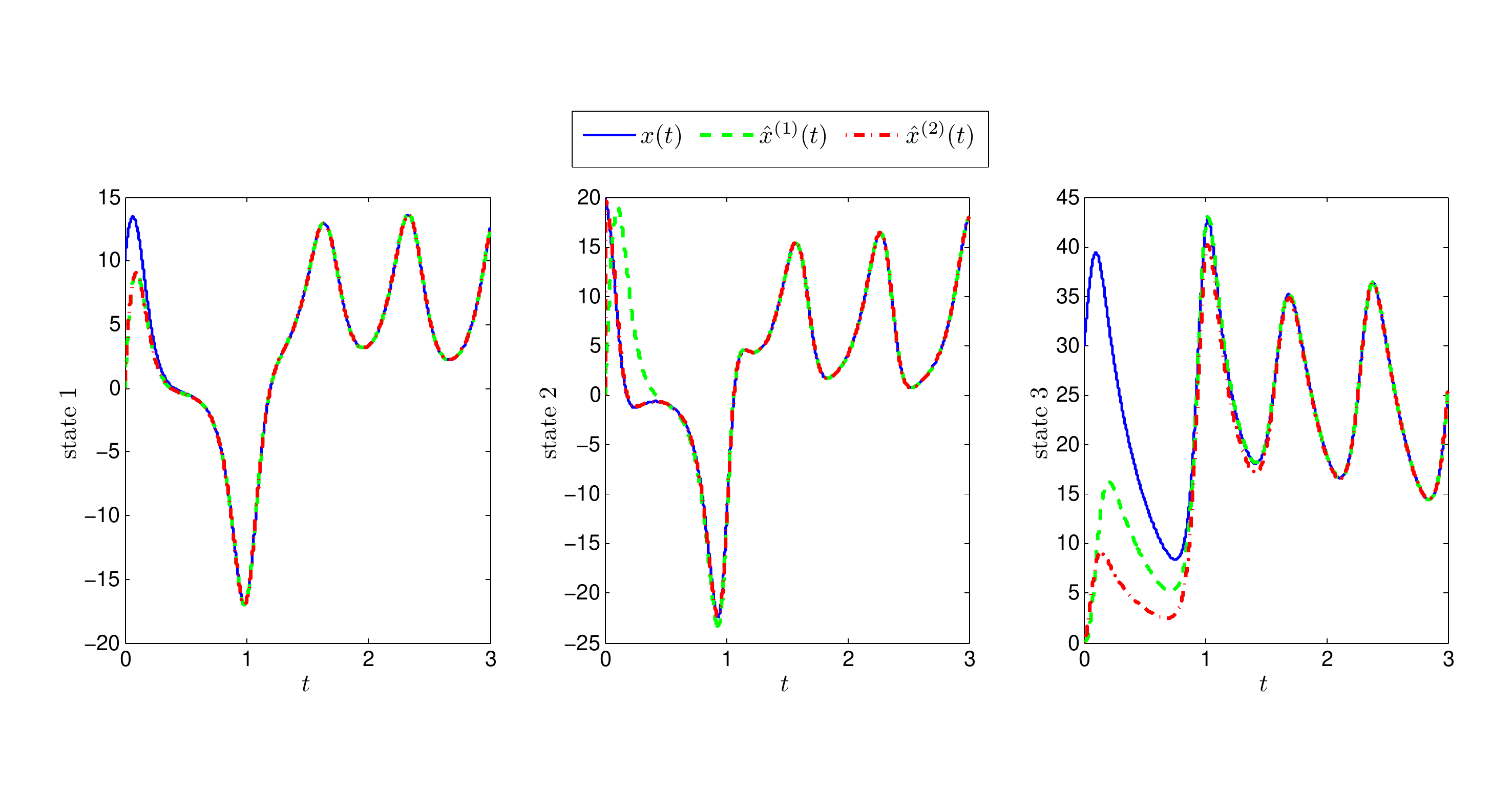}
\caption{\label{LorzFig} The Lorenz attractor is $(x(t))_{t \geq 0}$ with initial condition  $(10 \;20 \;30)^\top$. The observer with gain $L_5$ is $(x^{(1)}(t))_{t \geq 0}$; the observer with gain \eqref{eqn:Lgain} is $(x^{(2)}(t))_{t \geq 0}$. Both observers have initial condition $(0 \; 0 \; 0)^\top$. }
\end{center}
\end{figure*}
Observer design for the Lorenz attractor is considered in \cite{Krener}, where the co-ordinate transformation approach is used. This approach creates an observer which appears to converge experimentally, but the complexity of the co-ordinate transformation means that it is not possible to {\em prove} convergence. 

For the Lipschitz approach, suppose there exists $\gamma >0, P \in \mathcal{S}^n$ and $R \in \mathbb{R}^{p \times n}$ such that 
\begin{equation} \label{lipRicatti}
\left[ \begin{array}{cc} PA+A\tpose P  + RC + C\tpose R\tpose & P \\ (\ast) & -I/\gamma^2 \end{array} \right] \prec 0.
\end{equation}
It is easy to deduce (see e.g. \cite{Hu,Rajamani2}) that if $S$ is an invariant set for the state and the non-linearity satisfies the Lipschitz condition
\begin{equation*}
\| N(x)x -N(y)y \|_2 \leq \gamma \|x-y\|_2, \qquad x,y \in S,
\end{equation*}
then \eqref{LuOb}, for $L = P^{-1}R$, is a convergent observer. With respect to the Lorenz dynamics, the largest $\gamma>0$ satisfying \eqref{lipRicatti} is $\gamma = 2.67$. However, letting $x = (x_1,x_2,x_3), y = (y_1,x_2,x_3)$ implies that 
\begin{equation*}
\frac{ \|N(x)x-N(y)y\|_2}{\|x-y\|_2} = \sqrt{x_2^2 + x_3^2}. 
\end{equation*}
It is known that there exists $x$ in the range of the Lorenz attractor for which $\sqrt{x_2^2 + x_3^2} > \sqrt{1500}$ and hence, the Lipschitz approach (e.g. from \cite{Aboky}) cannot be used to construct a convergent observer for the Lorenz attractor.  

\subsection{Low order model for shear fluid flow} \label{examplesNineSec}

We consider observer design for the finite dimensional fluid flow model presented in \cite{Moehlis1}. The model is derived from the Navier-Stokes equations by the method of Galerkin projection. Before considering the example, we explain how this method necessarily results is a system of the form \eqref{nonlinsys} with nonlinear term satisfying the energy preserving property \eqref{energy}. 

The incompressible Navier-Stokes equations for a vector field $\bs{u} : \Omega \subset \mathbb{R}^3 \rightarrow \mathbb{R}^3$, are
\begin{align*}
\frac{\partial \bs{u}}{\partial t} + (\bs{u} \cdot \nabla) \bs{u} &= - \nabla p + \frac{1}{Re} \nabla^2 \bs{u} + \bs{f},\\
																					 \nabla \cdot \bs{u}  &= 0,
\end{align*}
where $p:\Omega \rightarrow \mathbb{R}$ represents the pressure, $\bs{f}:\Omega \rightarrow \mathbb{R}^3$ an external force and $Re$ the Reynold's number of the flow. No-slip boundary conditions ($\bs{u}_{|_{\partial \Omega}} = 0$) are also assumed. A common assumption \cite{Ravindran2,Ravindran} is that the flow field can be decomposed in the form 
\begin{equation} \label{flowExpand:eqn}
\bs{u}(x,t) = \sum_{i=1}^\infty a_i(t) \bs{u}_i(x),
\end{equation}
and a finite dimensional approximation of the flow obtained by considering the truncation $\bs{u}= \sum_{i=1}^n a_i \bs{u}_i$. 

A set of ordinary differential equations for the time-dependent coefficients $a_i$ can be obtained via the method of Galerkin projection (see e.g. \cite[pp.~129--154]{Berkooz}), leading to
\begin{equation} \label{TemporalODE}
\dot a_i  = \frac{\langle \bs{f}, \bs{u}_i \rangle}{\|\bs{u}_i\|^2} - \frac{\lambda_i}{Re}a_i + \sum_{j,k} \frac{a_j a_k}{\|\bs{u}_i\|^2} \langle (\bs{u}_j \cdot \nabla) \bs{u}_k, \bs{u}_i \rangle, \qquad i = 1,\ldots,n,
\end{equation}
where $\lambda_i >0$ are fixed constants. To remove the constant term from \eqref{TemporalODE} it is assumed that there exists a known stationary point $a = c$. Making the transformation $x = a-c$, the perturbations about $c$ have dynamics of the form \eqref{nonlinsys} with linear part
\begin{equation} \label{fluidLinear:eqn}
Ax := \frac{1}{Re}\Lambda x + N(c)x + N(x)c, \qquad x \in \mathbb{R}^n, 
\end{equation}
for $\Lambda  = -\text{diag} \left( \begin{array}{ccc} \lambda_1 & \cdots &\lambda_n\end{array} \right) \prec 0$ and nonlinear term 
\[
N(x) = \sum_{i=1}^n x_i \mathcal{Q}^{(i)}, \qquad \mathcal{Q}^{(i)} := \left( \|\bs{u}_i\|^{-2} \langle (\bs{u}_j \cdot \nabla) \bs{u}_k, \bs{u}_i \rangle \right)_{j,k=1}^n. 
\]
As a consequence of the incompressibility and no-slip assumptions, 
\[
\langle (\bs{u}_j \cdot \nabla) \bs{u}_k, \bs{u}_i \rangle = -  \langle (\bs{u}_i \cdot \nabla) \bs{u}_k, \bs{u}_j \rangle,
\]
which implies that the matrices $\mathcal{Q}^{(i)}$ are anti-symmetric. Hence, the nonlinearity $N$ satisfies \eqref{energy}. 

As a first step towards designing an observer, we construct an invariant set for the state dynamics. Although Proposition \ref{prop:smallestStateTrapping} can be applied, the particular structure  of the linear term \eqref{fluidLinear:eqn} implies that a natural invarient set can be easily constructed.


\begin{lem} \label{lem:ODETrappingSet}
Suppose that $(x(t))_{t \geq 0}$ satisfies \eqref{nonlinsys} with linear part of the form \eqref{fluidLinear:eqn}, for some $c \in \mathbb{R}^n$, and nonlinear part satisfying \eqref{energy}. Then there exists $r>0$ such that $B_r(-c)$ is invariant for $(x(t))_{t \geq 0}$. 
\end{lem}
\begin{proof}
Note that $A_{-c} = \frac{1}{Re} \Lambda \prec 0$. By Lemma \ref{lem:StateTrappingSet} it follows that $\|x+c\|_2^2$ is decreasing if 
\begin{equation*}
\frac{1}{Re}(x+c)\tpose \Lambda (x+c)\tpose + (x+c)\tpose \left(-\frac{1}{Re} \Lambda c + N(c)c \right) < 0. 
\end{equation*}
Standard algebraic manipulation shows that the set of $x \in \mathbb{R}^n$ for which the above inequality holds is equal to $\mathbb{R}^n \setminus E$, where $E$ is the ellipsoid
\begin{align*}
E := \bigg\{ x \in \mathbb{Re}^n : \sum_{i=1}^n  \lambda_i  \bigg(  x_i & - \frac{ Re}{2 \lambda_i}(\Lambda c + N(c)c)_i \bigg)^2  \\
& \leq \sum_{i=1}^n  \frac{Re^2}{4 \lambda_i} (\Lambda c + N(c)c)^2_i  \bigg\}
\end{align*}
The result follows if $r>0$ is chosen such that $E \subset B_r(-c)$. 
\end{proof}

We now consider observer design for a low order model of shear fluid flow. For brevity, we refer to \cite[pp.~7--8]{Moehlis1} for an explicit description of the model\footnote{With respect to the system parameters in \cite{Moehlis1}, we select $\alpha = 1/2, \beta = \pi/2$ and $\gamma =1$.} and note that all subsequent calculations are performed for Reynolds number $Re = 60$.

For this system, the first vector field $\bs{u}_1$ appearing the expansion \eqref{flowExpand:eqn} coincides with the laminar solution to the flow, implying that $c=\bs{e}_1$. Since $N(\bs{e}_1)\bs{e}_1 =0$, Lemma \ref{lem:ODETrappingSet} implies that $B_r(-\bs{e}_1)$ is invariant for the system if 
\begin{align*}
E &:= \left\{ x \in \mathbb{R}^n : \lambda_1 \left(  x_1 +  \frac{1}{2} \right)^2 + \sum_{i=1}^n \lambda_i x_i^2 \leq  \frac{ \lambda_1}{4}   \right\} \subset B_{r}(-\bs{e}_1).
\end{align*}
In particular, $B_{\lambda_1 / \lambda_{{\rm min}}}(-\bs{e}_1) = B_1(-\bs{e}_1)$ is an invariant set.  In fact, applying Proposition \ref{prop:smallestStateTrapping} with $\ker{Q} := \text{span}(\bs{e}_1)$ implies that
\[
B_{\xi}(-\xi \bs{e_1}), \qquad \xi:=0.9477
\]
is an invariant set. Hence, Proposition \ref{prop:smallestStateTrapping} provides a tighter invariant set that than the natural one derived from Lemma \ref{lem:ODETrappingSet}. 

We assume that the first six states of the system can be observed, i.e. 
\[
C := \left( \begin{array}{ccc}  I_6 & \vdots& \mathcal{O} \end{array} \right),
\]
where $\mathcal{O} \in \mathbb{R}^{6 \times 3}$ has all entries equal to zero. Let $Y$ be a $1$-norm ball of radius $3\xi$ such that $B_\xi(-\xi \bs{e}_1) \subset Y$. Applying Algorithm \ref{alg:ObserverCvgAlg_Iter} with $10$ iterations provides an observer gain $L_{10}$ and $P_{10} \in \mathcal{S}^n$ for which 
\[
P_{10}(A_y + L_{10}C) + (A_y + L_{10}C)P_{10} \prec 0, \qquad y \in Y.
\]
Consequently, Theorem \ref{thm:LocalObserver} implies that the observer is locally convergent. The complexity of the system makes it unlikely that a globally stable observer can be constructed using our methods. However, Figure \ref{NineStateFig} shows the unobserved states of the locally stable observer can be seen to converge to the true system state. 

\begin{figure*}[h]
\begin{center} 
\includegraphics[scale=0.53,trim=3.5cm 0.5cm 3cm 0cm]{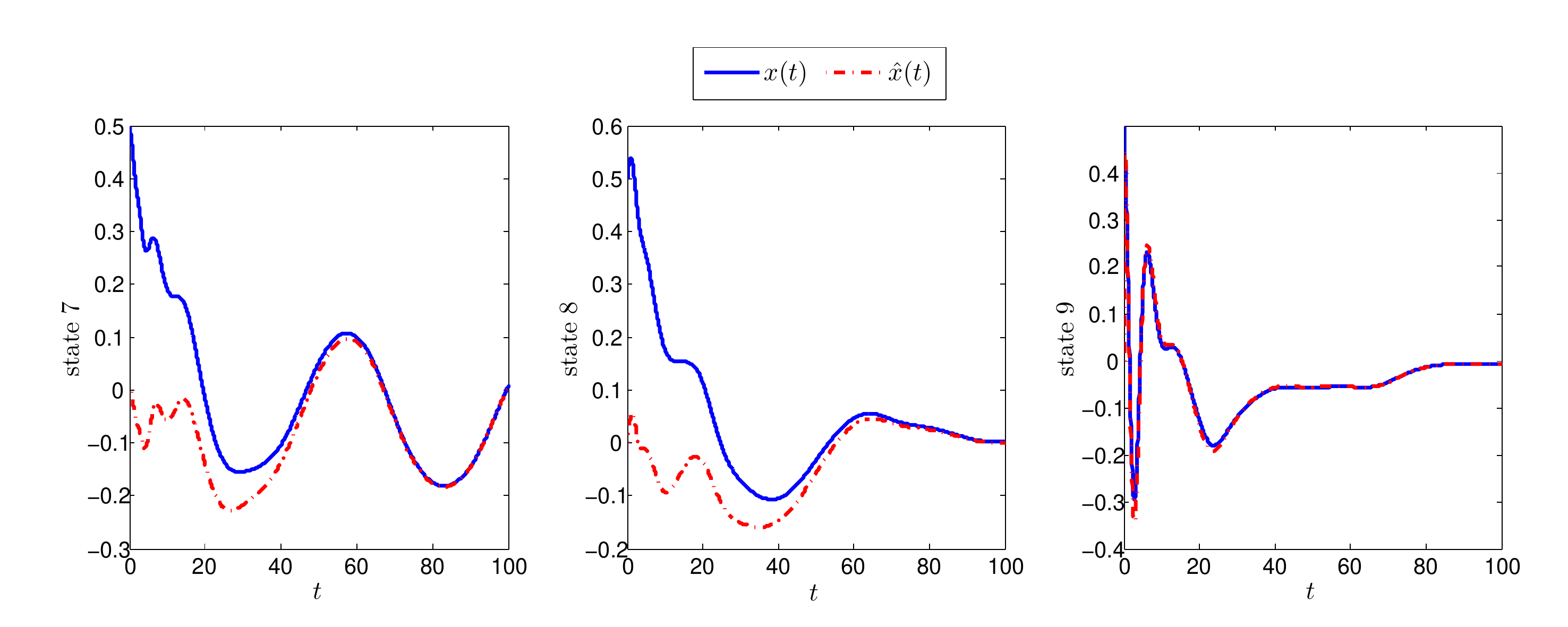}
\caption{\label{NineStateFig} The state fluid system with initial condition  $(-0.5 \; 0.5 \; 0.5 \; 0.5 \; 0.5 \; 0.5 \; 0.5 \; 0.5 \; 0.5)^\top$ is denoted $(x(t))_{t \geq 0}$. The observer with gain $L_{10}$ is $(\hat{x}(t))_{t \geq 0}$ with initial condition $(0 \; 0 \; 0 \; 0 \; 0 \; 0 \; 0 \; 0 \; 0)^\top$. }
\end{center}
\end{figure*}

\section{Conclusions}
A method of observer design has been presented for a class of non-linear systems whose non-linearity is energy preserving. Sufficient conditions, which can be verified by standard convex optimization methods, are given which imply either local or global observer convergence. The results are applied to create a globally convergent observer for the Lorenz attractor and a locally stable observer for a low order model of shear fluid flow. 

\bibliography{Observer}
\bibliographystyle{amsplain}

\end{document}